\documentclass[english,10pt]{article}
\usepackage[T1]{fontenc}
\usepackage[latin9]{inputenc}
\usepackage{amsmath,amsthm,amssymb,graphics,epsfig}
\usepackage{babel}

\theoremstyle{plain}
 \newtheorem{thm}{Theorem}
 \newtheorem{cor}{Corollary}
 
 \newtheorem{lem}{Lemma}
 
\theoremstyle{definition}
 \newtheorem{rem}{Remark}

\def\RR{\mathbb R}
\def\CC{\mathbb C}

\def\NN{\mathbb N}
\def\ZZ{\mathbb Z}

\begin{document}
\begin{center}
{\Large On generalized weighted Hilbert matrices}\\ $\,$ \\
Emmanuel Preissmann, Olivier L\'ev\^eque\\
Swiss Federal Institute of Technology - Lausanne, Switzerland
\end{center}

\begin{abstract}
In this paper, we study spectral properties of generalized weighted Hilbert matrices. In particular, we establish results on the spectral norm, determinant, as well as various relations between the eigenvalues and eigenvectors of such matrices. We also study the asymptotic behaviour of the spectral norm of the classical Hilbert matrix.
\end{abstract}

\section{Introduction}
The classical infinite Hilbert matrices
\begin{equation} \label{eq:hilb_matrices}
T_\infty = \begin{pmatrix} \ddots & \ddots & \ddots & \ddots & \ddots & \ddots \\
\ddots & 0 & -1 & -\frac{1}{2} & -\frac{1}{3} & \ddots \\
\ddots & 1 & 0 & -1 & -\frac{1}{2} & \ddots\\
\ddots & \frac{1}{2} & 1 & 0 & -1 & \ddots\\
\ddots & \frac{1}{3} & \frac{1}{2} & 1 & 0 & \ddots\\
\ddots & \ddots & \ddots & \ddots & \ddots & \ddots
\end{pmatrix}
\quad \text{and} \quad
H_\infty = \begin{pmatrix} 1 & \frac{1}{2} & \frac{1}{3} & \frac{1}{4} & \cdots \\
\frac{1}{2} & \frac{1}{3} & \frac{1}{4} & \frac{1}{5} & \ddots \\
\frac{1}{3} & \frac{1}{4} & \frac{1}{5} & \frac{1}{6} & \ddots \\
\frac{1}{4} & \frac{1}{5} & \frac{1}{6} & \frac{1}{7} & \ddots \\
\vdots & \ddots & \ddots & \ddots & \ddots 
\end{pmatrix}
\end{equation}
have been widely studied in the mathematical literature, for a variety of good reasons (see \cite{C83} for a nice survey of their astonishing properties). In this paper, we present results and conjectures on spectral properties of these matrices and related types of matrices. We first review known results in Section \ref{sec:survey} and then introduce new results in Section \ref{sec:new} on generalized weighted Hilbert matrices of the form
$$
b_{m,n}(\mathbf{x},\mathbf{c})=\begin{cases}
0 & \text{if } m=n\\
\dfrac{c_m \, c_n}{x_m-x_n} & \text{if } m \ne n
\end{cases}
$$
Our results can be summarized as follows. Theorem \ref{thm:maj} below states a surprising property of these matrices: their spectral norm depends monotonically in the absolute values of their entries, a property known a priori only for matrices with positive entries. A second important result (Theorem \ref{thm:det}) is that the determinant of such matrices are polynomials in the square of their entries. We prove next in Lemma \ref{lem:5} a key relation between the eigenvalues and eigenvectors of these matrices, that lead to a chain of nice consequences (among which Corollaries \ref{cor:1} and \ref{cor:2}). The present work finds its roots in the seminal paper of Montgomery and Vaughan \cite{MV73}, which initiated the study of generalized Hilbert matrices.

{\bf Notations.}
Let $p > 1$. In what follows, $\Vert \mathbf{y} \Vert_p$ denotes the $\ell^p$-norm of the vector $\mathbf{y} \in \CC^S$, i.e.
$$
\Vert \mathbf{y} \Vert_p:= \left( \sum_{k=1}^S \vert y_k\vert^p \right)^{1/p}
$$
and for an $S \times S$ matrix $M$, $\Vert M \Vert_p$ denotes the matrix norm induced by the above vector norm, i.e.
$$
\Vert M \Vert_p := \sup_{\Vert \mathbf{y} \Vert_p =1} \Vert M\mathbf{y} \Vert_p
$$
In the particular case $p=2$, the following simplified notation will be adopted:
$$
\Vert \mathbf{y} \Vert_2 = \Vert \mathbf{y} \Vert  \text{ (Euclidean norm)} \quad \text{and} \quad \Vert M \Vert_2=\Vert M \Vert
$$
Notice in addition that when $M$ is normal (i.e.~when $MM^*=M^*M$, where $M^*$ stands for the complex-conjugate transpose of the matrix $M$), the above norm is equal to the spectral norm of $M$, i.e.
$$
\Vert M \Vert = \sup\{ |\lambda| \; : \; \lambda \in \text{Spec}(M)\}
$$

%%%%%%%%%%%%%%%%%%%%%%%%%%%%%%%%%%%%%%%%%%%%%%%%%%%%%%%%%%%%%%%%%
%%%%%%%%%%%%%%%%%%%%%%%%%%%%%%%%%%%%%%%%%%%%%%%%%%%%%%%%%%%%%%%%%
%%%%%%%%%%%%%%%%%%%%%%%%%%%%%%%%%%%%%%%%%%%%%%%%%%%%%%%%%%%%%%%%%

\section{A survey of classical results and conjectures} \label{sec:survey}

\subsection{Hilbert's inequalities}
The infinite-dimensional matrices presented in \eqref{eq:hilb_matrices} are two different versions of the classical Hilbert matrix. Notice first that $T_\infty$ is a Toeplitz matrix (i.e.~a matrix whose entry $n,m$ only depends on the difference $m-n$), while $H_\infty$ is a Hankel matrix (i.e.~a matrix whose entry $n,m$ only depends on the sum $n+m$).

The original Hilbert inequalities state (see \cite[p.~212]{HLP52}) that for ${\mathbf u},{\mathbf v} \in \ell^2(\ZZ;\CC)$ (resp.~in ${\mathbf u},{\mathbf v} \in \ell^2(\NN;\CC)$) with $\Vert {\mathbf u} \Vert = \Vert {\mathbf v} \Vert =1$,
$$
\Big| \sum_{m,n \in \ZZ} u_m \, (T_\infty)_{m,n} \, v_n  \Big| \le \pi \quad \text{resp.} \quad \Big| \sum_{m,n \in \NN} u_m \, (H_\infty)_{m,n} \, v_n \Big| \le \pi
$$
where $\pi$ cannot be replaced by a smaller constant\footnote{Notice that Hilbert had proved originally these inequalities with $2\pi$ instead of $\pi$; the optimal constant was found later by Schur.}. This is saying that $T_\infty$ (resp.~$H_\infty$) is a bounded operator in $\ell^2(\ZZ;\CC)$ (resp.~in $\ell^2(\NN;\CC)$), with norm equal to $\pi$.

Hardy, Littlewood and P\'olya obtained an explicit expression for $\Vert H_\infty \Vert_p$ in \cite[p.~227]{HLP52}, for all values of $p>1$:
$$
\Vert H_\infty \Vert_p = \frac{\pi}{\sin(\pi/p)}, \quad p>1
$$
and Titchmarsh proved in \cite{T26} that $\Vert T_\infty \Vert_p< \infty$. Also, $\Vert T_\infty \Vert_p$ is clearly greater than or equal to $\Vert H_\infty \Vert_p$, as $H_\infty$ may be seen as the lower-left corner of $T_\infty$ (up to a column permutation), but no exact value is known for it (except in the case where $p=2^n$ or $p=2^n/(2^n-1)$ for some integer $n \ge 1$; see \cite{L07,L09} for a review of the subject).

Consider now the corresponding finite-dimensional matrices $T_R$ and $H_R$ of size $R \times R$:
$$
(T_R)_{m,n}=\begin{cases} 0 & \text{if } m=n\\ \dfrac{1}{m-n} & \text{if } m \ne n \end{cases}
\quad \text{and} \quad (H_R)_{m,n}=\frac{1}{m+n-1} \quad 1 \le m,n \le R
$$
The above Hilbert inequalities imply that for every integer $R \ge 1$, both
\begin{equation} \label{eq:Hilbert}
\Vert T_R \Vert < \pi \quad \text{and} \quad \Vert H_R \Vert < \pi
\end{equation}
Clearly also, both $\Vert T_R \Vert$ and $\Vert H_R \Vert$ increase as $R$ increases, and $\lim_{R\to\infty} \Vert T_R \Vert = \lim_{R\to\infty} \Vert H_R \Vert =\pi$.

A question of interest is the convergence speed of $\Vert H_R \Vert$ and $\Vert T_R \Vert$ towards their common limiting value~$\pi$. Observing that up a column permutation, $H_R$ can be seen as the lower-left corner of $T_{2R+1}$, we see that $\Vert H_R \Vert \le \Vert T_{2R+1} \Vert$ for every integer $R\ge 1$. This hints at a slower convergence speed for the matrices $H_R$ than for the matrices $T_R$. Indeed, Wilf et de Bruijn (see \cite{W70}) have shown that
$$
\pi - \Vert H_R \Vert \sim \frac{\pi^{5}}{2 \, (\log(R))^2} \quad \text{as } R \to \infty
$$
whereas there exist $a,b>0$ such that 
\begin{equation} \label{eq:conv_speed}
\frac{a}{R} < \pi - \Vert T_R \Vert < \frac{b \, \log(R)}{R} \quad (R \ge 2)
\end{equation}
We will prove these inequalities at the end of the present paper. The lower bound has already been proved by H.~Montgomery (see \cite{www}) and it has been conjectured in \cite{P85} and independently by Montgomery, that the upper bound in the previous inequality is tight, i.e.~that
$$
\pi - \Vert T_R \Vert \sim \frac{c \, \log R}{R} \quad \text{as } R \to \infty
$$
We also provide some numerical indication of this conjecture at the end of the paper.

%%%%%%%%%%%%%%%%%%%%%%%%%%%%%%%%%%%%%%%%%%%%%%%%%%%%%%%%%%%%%%%%%%

\subsection{Toeplitz matrices and Grenander-Szeg\"o's theorem} \label{sec:toeplitz}

We review here the theory developed by Grenander and Szeg\"o in \cite{GZ84} for analyzing the asymptotic spectrum of Toeplitz matrices. In particular, we cite below their result on the convergence speed of the spectral norm of such matrices.

Let $(c_r, \, r \in \ZZ)$ be a sequence of complex numbers such that
\begin{equation} \label{eq:summable}
\sum_{r \in \ZZ} |c_r|<\infty
\end{equation}
and let us define the corresponding function, or {\em symbol}:
$$
f(x) = \sum_{r \in \ZZ} c_r \exp(irx), \quad x \in [0,2\pi]
$$
Because of assumption made on the Fourier coefficients $c_r$, $f$ is a continuous function such that $f(0)=f(2\pi)$ (equivalently, $f$ can be viewed as a continuous $2\pi$-periodic function on $\RR$).

Let now $C_R$ be the $R \times R$ matrix defined a
$$
(C_R)_{m,n} = c_{m-n}, \quad 1 \le m,n \le R
$$
The following fact can be verified by a direct computation: for any vector $\mathbf{u} \in \CC^R$ such that $\Vert \mathbf{u} \Vert^2 = \sum_{1\le n \le R} |u_n|^2=1$,
\begin{equation} \label{eq:fourier}
\mathbf{u}^* C_R \mathbf{u} = \int_0^{2\pi} f(x) \, |\phi(x)|^2 \, dx
\end{equation}
where $f(x)$ is the above defined function and $\phi(x)=\frac{1}{\sqrt{2\pi}} \sum_{1\le n \le R} u_n \, \exp(i(n-1)x)$.

Let us now assume that $C_R$ is a normal matrix (i.e.~$C_R C_R^*=C_R^* C_R$); this is the case e.g.~when $f$ is a real-valued function (in which case $C_R$ is Hermitian, i.e.~$C_R^*=C_R$). As $\Vert \mathbf{u} \Vert=1$, we also have $\int_0^{2\pi} |\phi(x)|^2 \, dx =1$, which implies that
$$
\Vert C_R \Vert \le \sup_{x \in [0,2\pi]} |f(x)| =: M
$$
for any integer $R \ge 1$. Grenander and Szeg\"o proved in \cite[p.~72]{GZ84} the following refined statement of the convergence speed of the spectral norm. If $f$ is twice continuously differentiable, admits a unique maximum in $x_0$ and is such that $f''(x_0) \ne 0$, then
$$
M - \Vert C_R \Vert \sim f(x_0) - f \left(x_0 + \frac{\pi}{R} \right) \sim \frac{\pi^2 \, |f''(x_0)|}{2 R^2}
$$
as $R \to \infty$.

The above theorem does however not apply to Hilbert matrices of the form $T_R$, as the harmonic series $\sum_{r \ge 1} \frac{1}{r}$ 
diverges, so condition \eqref{eq:summable} is not satisfied. Correspondingly, the symbol associated to these matrices is the function
$$
f(x) = \sum_{r \ge 1} \frac{-\exp(irx)+\exp(-irx)}{r} = -2 i \sum_{r \ge 1} \frac{\sin(rx)}{r} = i (x-\pi), \quad x \in \, ]0,2\pi[
$$
while by Dirichlet's theorem, $f(0)=f(2\pi)=0$; $f$ is therefore discontinuous, but relation \eqref{eq:fourier} still holds in this case and allows to deduce Hilbert's inequality: 
$$
\Vert T_R \Vert \le \sup_{x \in [0,2\pi]} |f(x)| = \pi
$$
However, relation \eqref{eq:fourier} alone does not allow to conclude on the convergence speed towards $\pi$.

In general, the problem of evaluating the convergence speed of the spectral norm is a difficult one when $f$ attains its maximum at a point of discontinuity. An interesting matrix entering into this category has been the object of a detailed study by Slepian in \cite{S78} (see also \cite{V93} for a recent exposition of the problem\footnote{We would like to thank Ben Adcock for pointing out this interesting reference to us.}). It is the so-called {\em prolate matrix}, defined as
$$
(P_R)_{m,n} = p_{m-n}, \quad 1 \le m,n \le R, \quad \text{where } p_r = \begin{cases} \frac{\sin(2 \pi w r)}{r} & \text{if } r \ne 0\\ 2 \pi w & \text{if }r=0 \end{cases}
$$
where $0<w<\frac{1}{2}$ is a fixed parameter. Here again, we see that condition \eqref{eq:summable} is not satisfied. The symbol associated to this matrix is the function
$$
f_w(x) = \sum_{r \in \ZZ} p_r \, \exp(i r x) = 2 \pi w + 2 \sum_{r \ge 1} \frac{\sin(2 \pi w r)}{r} \, \cos(rx)
= \pi \, 1_{[0,2\pi w] \cup [2\pi(1-w),2\pi]}(x)
$$
for all $x \in [0,2\pi]\backslash\{2\pi w,2 \pi(1-w)\}$. In this case, we again have for any integer $R \ge 1$
$$
\Vert P_R \Vert < \sup_{x \in [0,2\pi]} |f_w(x)| = \pi \quad \text{and} \quad \lim_{R \to \infty} \Vert P_R \Vert = \pi
$$
It is moreover shown in \cite{S78} that for all $0 < \omega < \frac{1}{2}$, there exist $c_w, d_w>0$ (where both $c_w$ and $d_w$ are given explicitly in \cite{V93}) such that
$$
\pi - \Vert P_R \Vert \sim c_w \, \sqrt{R} \, \exp(- d_w R)
$$
We see here that even though the function $f_w$ is discontinuous, the convergence speed is exponential, as opposed to polynomial in the case of a smooth symbol. Of course, the situation here is quite particular, as the function $f_w$ has a plateau at its maximum value, which is not the case for the Hilbert matrix $T_R$.

%%%%%%%%%%%%%%%%%%%%%%%%%%%%%%%%%%%%%%%%%%%%%%%%%%%%%%%%%%%%%%%%

\subsection{Generalized weighted Hilbert matrices}

Let $\mathbf{x}=(x_1,\ldots,x_R)$ be a vector of distinct real numbers, $\mathbf{c}=(c_1,\ldots,c_R)$ be another vector of real numbers, and let us define the matrices $A(\mathbf{x})$ and $B(\mathbf{x},\mathbf{c})$ of size $R \times R$ as
\begin{equation} \label{eq:ar}
a_{m,n}(\mathbf{x})=\begin{cases}
0 & \text{if } m=n\\
\dfrac{1}{x_m - x_n } & \text{if } m\ne n
\end{cases}
\end{equation}
and
\begin{equation} \label{eq:br}
b_{m,n}(\mathbf{x},\mathbf{c})=\begin{cases}
0 & \text{if } m=n\\
\dfrac{c_m \, c_n}{x_m-x_n} & \text{if }  m \ne n
\end{cases}
\end{equation}
If there is no risk of confusion, the matrices $A(\mathbf{x})$ and $B(\mathbf{x},\mathbf{c})$ will be denoted as $A$ and $B$, respectively.

%%%%%%%%%%%%%%%%%%%%%%%%%%%%%%%%%%%%%%%%%%%%%%%%%%%%%%%%%%%%%%%%

\subsection{A result and a conjecture by Montgomery and Vaughan}

In order to motivate the study of generalized Hilbert matrices, let us mention here both a result and a conjecture made by Montgomery and Vaughan. The next section will be devoted to applications of these questions.

In \cite{MV73}, Montgomery and Vaughan showed the following result. If $A(\mathbf{x})$ is the $R \times R$ matrix with entries as in \eqref{eq:ar}, then
$$
\Vert A(\mathbf{x}) \Vert \le \frac{\pi}{\delta}
$$
where $\delta = \inf_{1 \le m,n \le R, m \ne n} |x_n-x_m|$.

They further showed that if $B(\mathbf{x},\mathbf{c})$ is the $R \times R$ matrix with entries as in \eqref{eq:br} with $c_n=\sqrt{\delta_n}$ and $\delta_n=\min_{1 \le m \le R, \, m \ne n} |x_m-x_n|$, then
\begin{equation} \label{eq:MV}
\Vert B(\mathbf{x},\mathbf{c}) \Vert \le \frac{3 \pi}{2}
\end{equation}
They also conjectured that the tightest upper bound is actually $\Vert B(\mathbf{x},\mathbf{c}) \Vert \le \pi$. In \cite{P84}, Montgomery and Vaughan's result was improved to $\Vert B(\mathbf{x},\mathbf{c}) \Vert \le \frac{4 \pi}{3}$, but the conjecture remains open so far.

%%%%%%%%%%%%%%%%%%%%%%%%%%%%%%%%%%%%%%%%%%%%%%%%%%%%%%%%%%%%%%%%

\subsection{Applications}

\subsubsection{Large sieve inequalities}

Let $x_1,\ldots,x_R$ be real numbers which are distinct modulo $1$. Let also $\Vert t\Vert$ be the distance from the real number $t$ to the closest integer and let
$$
\delta:=\min_{r,s, \, r\ne s}\Vert x_r-x_s\Vert \quad \text{and} \quad \delta_r:=\min_{s, \, s\ne r}\Vert x_r-x_s \Vert
$$ 
For $(a_n)_{M+1\le n \le M+N}$  an arbitrary sequence of complex numbers, we write
$$
S(x):=\sum_{M+1\le n\le M+N} a_n \exp(2 \pi i nx)
$$
A large sieve inequality has the generic form
$$
\sum_{1 \le r \le R} |S(x_r)|^2 \le \Delta(N,\delta) \, \sum_{M+1 \le n\le M+N} |a_n|^2
$$
Using Hilbert's inequality \eqref{eq:Hilbert}, one can show that the previous inequality holds with $\Delta(N,\delta) = N + \delta^{-1} - 1$. Equivalently, this says that
$$
\text{if } B:= \left\{\exp(2\pi i n x_r) \right\}_{M+1 \le n \le M+N, \, 1 \le r \le R}, \quad \text{then }\Vert B\Vert^2 \le \Delta(N,\delta)
$$
Besides, generalized Hilbert inequalities of the type \eqref{eq:MV} are particularly useful when studying irregularly spaced $x_r$ (such as Farey sequences). These generalized inequalities allow to prove the following refined large sieve inequality:
$$
\sum_{1\le r\le R} \left(N+\frac{3}{2} \delta_r^{-1} \right)^{-1} \vert S(x_{r})\vert^2 \le \sum_{M+1 \le n\le M+N} |a_n|^2
$$
This last result is useful for arithmetic applications, as it allows e.g.~to show (see \cite{MV73}) that $\pi(M+N)-\pi(M)\le 2\pi(N)$ ($\pi(N)$ being the number of primes smaller than or equal to $N$), whereas the inequality $\pi(M+N)-\pi(M) \le \pi(N)$ stands as a conjecture so far.

Another important application of large sieve inequalities is the Bombieri-Vinogradov theorem (see for instance \cite{BFI86}), which is related to various conjectures on the distribution of primes.

\subsubsection{Other Hilbert inequalities}

In \cite{MV74}, Montgomery and Vaughan study variants of Hilbert's inequality (with for instance $\frac{1}{x_r-x_s}$ replaced by $\csc(x_r-x_s)$), which allow them to show the following result: if $\sum_{n \ge 1} n |a_n|^2 < \infty$, then
$$
\int_0^T \Big| \sum_{n \ge 1} a_n n^{-it} \Big|^2 \, dt = \sum_{n\ge1} | a_n|^2 \, (T+O(n))
$$
The key idea behind the proof of the main result in \cite{MV74} is the following identity:
$$
\csc(x_k-x_l) \, \csc(x_l-x_m) = \csc(x_k-x_m) \, (\cot(x_k-x_l) + \cot(x_l-x_m))
$$
which is of the same type as our relation \eqref{eq:a} below. Building on this, a further generalization of Hilbert's inequalities has been performed in \cite{P87}, where the following functional equations are solved:
$$
\frac{1}{\theta(x)\theta(y)}=\Psi(x)-\Psi(y)+\frac{\phi(x-y)}{\theta(x-y)} \quad \text{and} \quad \frac{1}{\theta(x)\theta(y)}=\frac{\sigma(x)-\sigma(y)}{\theta(x-y)}+\tau(x)\tau(y) \quad \text{(with }\tau(0)=0)
$$

%%%%%%%%%%%%%%%%%%%%%%%%%%%%%%%%%%%%%%%%%%%%%%%%%%%%%%%%%%%%%%%%
%%%%%%%%%%%%%%%%%%%%%%%%%%%%%%%%%%%%%%%%%%%%%%%%%%%%%%%%%%%%%%%%
%%%%%%%%%%%%%%%%%%%%%%%%%%%%%%%%%%%%%%%%%%%%%%%%%%%%%%%%%%%%%%%%

\section{New results} \label{sec:new}

%%%%%%%%%%%%%%%%%%%%%%%%%%%%%%%%%%%%%%%%%%%%%%%%%%%%%%%%%%%%%%%%
 
\subsection{Spectral norm of $B(\mathbf{x},\mathbf{c})$} \label{sec:norm}
We establish below a monotonicity result regarding the spectral norm of matrices of the type $B(\mathbf{x},\mathbf{c})$.

\begin{thm} \label{thm:maj}
If $\mathbf{x},\mathbf{x}',\mathbf{c},\mathbf{c'}$ are vectors of real numbers such that 
$$
\vert b_{m,n}(\mathbf{x},\mathbf{c})\vert \le \vert b_{m,n}(\mathbf{x'},\mathbf{c'})\vert \quad \text{for } 1 \le m,n \le R
$$
then 
\begin{equation} \label{eq:maj}
\Vert B(\mathbf{x},\mathbf{c}) \Vert \le \Vert B(\mathbf{x'},\mathbf{c'}) \Vert
\end{equation}
\end{thm}

\begin{rem}
Notice for matrices $Y, Z$ with positive entries, it holds that if $0 \le y_{m,n} \le z_{m,n}$ for all $m,n$, then $\Vert Y \Vert \le \Vert Z \Vert$. Consider indeed the normalized eigenvector $u$ corresponding to the largest eigenvalue of $Y^*Y$: as $Y^*Y$ has positive entries, $u$ is also positive, so $\Vert Y \Vert = \Vert Y u \Vert \le \Vert Z u \Vert \le \Vert Z \Vert$. The above result states that a similar result holds for matrices of the form $B(\mathbf{x},\mathbf{c})$, even though these do not have positive entries.
\end{rem}

The remainder of the present section is devoted to the proof of Theorem \ref{thm:maj}, which we decompose into a sequence of lemmas.

First of all, let us observe that the numbers $a_{m,n}=1/(x_m-x_n)$ satisfy
\begin{equation} \label{eq:a}
a_{k,l} \, a_{l,m} = a_{k,m} \, (a_{k,l}+a_{l,m}) \quad \text{for } k,l,m \text{ distinct}
\end{equation}
This relation will be of primary importance in the sequel.

\begin{lem} \label{lem:1}
If $k$ is a positive integer and $1 \le n \le R$, then, denoting as $B_{-n}$ the matrix $B$
with $n^{th}$ row and column removed, we obtain
\begin{equation} \label{eq:s}
S := \sum_{\substack{1 \le l,m \le R \\ l \ne n, \, m \ne n, \, l \ne m}} b_{n,l} \, b_{m,n} \, (B_{-n}^k)_{l,m}=0
\end{equation}
\end{lem}

\begin{proof}
Using \eqref{eq:a}, we obtain
$$
S = \sum_{\substack{1 \le l,m \le R \\ l \ne n, \, m \ne n, l \ne m}} c_l \, c_m \, c_n^2 \, a_{m,n} \, a_{n,l} \, (B_{-n}^k)_{l,m}
= \sum_{\substack{1 \le l,m \le R \\ l \ne n, \, m \ne n, \, l \ne m}} c_l \, c_m \, c_n^2 \, a_{m,l} \, (a_{m,n} + a_{n,l}) \, (B_{-n}^k)_{l,m}= S_1 + S_2
$$
with
$$
S_1 = \sum_{\substack{1 \le l,m \le R \\ l \ne n, \, m \ne n, \, l \ne m}} c_l \, c_m \, c_n^2 \, a_{m,l} \, a_{m,n} \, (B_{-n}^k)_{l,m}
= \sum_{\substack{1 \le l,m \le R \\ l \ne n, \, m \ne n, \, l \ne m}} c_n^2 b_{m,l} \, a_{m,n} \, (B_{-n}^k)_{l,m} = \sum_{\substack{1 \le m \le R\\ m \ne n}} c_n^2 \, a_{m,n} \, (B_{-n}^{k+1})_{m,m}
$$
and
$$
S_2 =  \sum_{\substack{1 \le l,m \le R \\ l \ne n, \, m \ne n, \, l \ne m}} c_l \, c_m \, c_n^2 \, a_{m,l} \, a_{n,l} \, (B_{-n}^k)_{l,m} = \sum_{\substack{1 \le l \le R\\ l \ne n}} c_n^2 \, a_{n,l} \, (B_{-n}^{k+1})_{l,l} = - \sum_{\substack{1 \le l \le R \\ l \ne n}} c_n^2 \, a_{l,n} \, (B_{-n}^{k+1})_{l,l}=-S_1
$$
as $A$ is antisymmetric.
\end{proof}

\begin{lem} \label{lem:2}
Let $1 \le n \le R$ and $k\ge 2$ be an integer. Then
$$
(B^k)_{n,n} = \sum_{0 \le r \le k-2} \; \sum_{\substack{1 \le l,m \le R\\ l \ne n, \, m \ne n}} b_{n,l} \, \left(B_{-n}^r \right)_{l,m} \, b_{m,n} \, \left(B^{k-r-2} \right)_{n,n}
= -\sum_{0 \le r \le k-2} \; \sum_{1 \le l \le R} \, b_{n,l}^2 \, (B_{-n}^r)_{l,l} \, (B^{k-r-2})_{n,n}
$$
\end{lem}

\begin{proof}
Notice first that
$$
(B^k)_{n,n} = \sum_{1\le n_1,\ldots,n_{k-1} \le R} b_{n,n_1} \, b_{n_1,n_2} \cdots b_{n_{k-2},n_{k-1}} \, b_{n_{k-1},n}
$$
As $b_{n,n}=0$, we may consider $n_1, n_{k-1} \ne n$ in the above sum. For each $(n_1,\ldots,n_{k-1})$, define
$$
s = \inf\{t \in \{2,\ldots,k\} \; |  \, n_1 \ne n,\ldots,n_{t-1} \ne  n, n_t=n\}
$$
(where by convention, we set $n_k=n$ in the above definition). Ordering the terms in the above sum according to the value of $s$, we obtain
\begin{eqnarray*}
(B^k)_{n,n} & = & \sum_{2 \le s \le k} \; \sum_{n_1, \, n_{s-1} \ne n} b_{n,n_1} \, \left(B^{s-2}_{-n} \right)_{n_1,n_{s-1}} \, b_{n_{s-1},n}
\, \left(B^{k-s} \right)_{n,n}\\
& = & \sum_{0 \le r \le k-2} \; \sum_{n_1, \, n_{r+1} \ne n} b_{n,n_1} \, \left(B^r_{-n} \right)_{n_1,n_{r+1}} \, b_{n_{r+1},n}
\, \left(B^{k-r-2} \right)_{n,n}
\end{eqnarray*}
which is the first equality in the lemma. The second one follows from \eqref{eq:s} and the fact that $B$ is antisymmetric.
\end{proof}

\begin{lem} \label{lem:3}
Let $1\le n\le R$ and $k\ge 2$ be an integer. Then the following holds

\begin{tabular}{l}
- if $k$ is odd, then $(B^k)_{n,n}=0$\\
- if $k$ is even, then $(-1)^{\frac{k}{2}} \, (B^k)_{n,n}$ is a polynomial in $(b_{l,m}^2, \; 1 \le l < m \le R)$ with positive coefficients.
\end{tabular}
\end{lem}

\begin{proof}
Since $B$ is antisymmetric, the first statement is obvious. The second one follows by induction from Lemma \ref{lem:2}.
\end{proof}

{\em Proof of Theorem \ref{thm:maj}.}
Observe that since the matrix $iB$ is Hermitian, it has $R$ real eigenvalues $\mu_1,\ldots,\mu_R$ corresponding to an orthonormal basis of eigenvectors, so
$$
\Vert B \Vert = \max_{1 \le r \le R} |\mu_r|
$$
and for a positive integer $k$
$$
\mathrm{Tr}(B^{2k})=\sum_{1\le r\le R} (-1)^k \mu_r^{2k}
$$
Therefore, we obtain
$$
\left\Vert B\right\Vert = \lim_{k\to\infty} \left( (-1)^k \,  \mathrm{Tr}(B^{2k}) \right)^{\frac{1}{2k}} 
$$
and the theorem follows from Lemma \ref{lem:3}. \hfill $\square$

%%%%%%%%%%%%%%%%%%%%%%%%%%%%%%%%%%%%%%%%%%%%%%%%%%%%%%%%%%%%%%%%%

\subsection{Determinant of $B(\mathbf{x},\mathbf{c})$} \label{sec:det}

The following result shows that the determinant of $B(\mathbf{x},\mathbf{c})$ is a polynomial in $b_{l,m}^2$.

\begin{thm} \label{thm:det}
-If $R$ is odd, then $\det(B(\mathbf{x},\mathbf{c}))=0$.\\
-If $R=2T$ is even, then
\begin{equation} \label{eq:det}
\det(B(\mathbf{x},\mathbf{c}))=\prod_{k=1}^R c_k^2 \sum_{(m_i,n_i)_1^T \in E} \; \prod_{i=1}^T a_{m_i,n_i}^2 = \sum_{(m_i,n_i)_1^T \in E} \; \prod_{i=1}^T b_{m_i,n_i}^2
\end{equation}
where
$$
E:=\{(m_i,n_i)_1^T \; \vert \; \cup_{i=1}^T\{m_i,n_i\}=\{1,\ldots,R\} \; \mathrm{and} \; m_i<n_i, \;\forall i,\; m_1 < \ldots < m_T\}
$$
\end{thm}

Let us first establish the following lemma.

\begin{lem} \label{lem:4}
Let $l$ be an integer, with $3 \le l \le R$. Denoting by $\mathcal{S}_l$ the set of permutations of $\{1,\ldots,l\}$, we have
\begin{equation} \label{eq:s2}
S := \sum_{\sigma\in\mathcal{S}_l} a_{\sigma(1),\sigma(2)} \, a_{\sigma(2),\sigma(3)} \cdots a_{\sigma(l-1),\sigma(l)} \, a_{\sigma(l),\sigma(1)} = 0
\end{equation}
\end{lem}

\begin{proof}
Let us define
$$
S_1 := \sum_{\sigma\in\mathcal{S}_l} a_{\sigma(1),\sigma(2)} \, a_{\sigma(2),\sigma(3)} \cdots a_{\sigma(l-1),\sigma(1)} \, a_{\sigma(l-1),\sigma(l)}
$$
and
$$
S_2 := \sum_{\sigma\in \mathcal{S}_l} a_{\sigma(1),\sigma(2)} \, a_{\sigma(2),\sigma(3)} \cdots a_{\sigma(l-1),\sigma(1)} \, a_{\sigma(l),\sigma(1)}
$$
By \eqref{eq:a}, we have $S=S_1+S_2$. Let now $\tau \in \mathcal{S}_l$ be the permutation defined  by $\tau(1)=l-1,\tau(2)=1,\tau(3)=2,\ldots,\tau(l-1)=l-2,\tau(l)=l$. We obtain
\begin{eqnarray*}
S_2 & = & \sum_{\sigma \in \mathcal{S}_l} a_{\sigma\tau(1),\sigma\tau(2)} \, a_{\sigma\tau(2),\sigma\tau(3)} \cdots a_{\sigma\tau(l-1),\sigma\tau(1)} \, a_{\sigma\tau(l),\sigma\tau(1)}\\
& = & \sum_{\sigma\in S_l} a_{\sigma(l-1),\sigma(1)} \, a_{\sigma(1),\sigma(2)} \cdots a_{\sigma(l-2),\sigma(l-1)} \, a_{\sigma(l),\sigma(l-1)} = - S_1
\end{eqnarray*}
which completes the proof.
\end{proof}

{\em Proof of Theorem \ref{thm:det}.} 
By definition,
$$
\det(B)=\sum_{\sigma \in \mathcal{S}_R} \varepsilon(\sigma) \, \prod_{1 \le n \le R} a_{n,\sigma(n)}  \, c_n^2
$$
Now, every permutation $\sigma$ is a product of cycles, say $F_1, \, F_2, \ldots, F_k$. Let us denote by 
$n_1,n_2,\ldots,n_k$ the respective cardinalities of these cycles and let us set
$$
S(F_i) := \sum_{s_1,s_2,\ldots,s_{n_i} \vert \{ s_1,s_2,\ldots,s_{n_i} \} = F_i} a_{s_1,s_2} \, a_{s_2,s_3} \cdots a_{s_{n_i-1},s_{n_i}} \, a_{s_{n_i},s_1}
$$
In the above expression for $\det(B)$, the contribution of the permutations having these sets as support for their cycles is
$$
(-1)^{n_1+n_2+\ldots+n_k-k} \, \prod_{i=1}^k \, S(F_i) \, \prod_{r=1}^R c_r^2
$$
By \eqref{eq:s2} and the fact that the main diagonal is zero, a non-zero contribution can therefore only occur when all cycles are of cardinality $2$, which proves the theorem. \hfill $\square$

\begin{rem}
The above statement allows to recover part of the conclusion of Lemma \ref{lem:3}. First notice that by Theorem \ref{thm:det} and for all $J \subset \{1,\ldots,R\}$,
$\det(B_J)$, where $B_J=(b_{l,m})_{l,m \in J}$, is also a polynomial in $b_{l,m}^2$. Define then
$$
\sigma_k = \sum_{\substack{J \subset \{1,\ldots,R\} \\ |J|=k}} \prod_{i \in J} \lambda_i
$$
where $\lambda_1, \ldots, \lambda_R$ are the eigenvalues of $B$. Notice that
\begin{equation} \label{eq:ev_det}
\sigma_k = \sum_{\substack{J \subset \{1,\ldots,R\} \\ |J|=k}} \det(B_J)
\end{equation}
Indeed, let $P$ be the polynomial defined as $P(x)=\prod_{1 \le i \le R} (x-\lambda_i)$. We observe that on one hand, the matrix-valued version of this polynomial is given by 
$$
P(x) = \prod_{1 \le i \le R} (x-\lambda_i I) = x^R + \sum_{k=1}^R x^{R-k} \, (-1)^k \sum_{\substack{J \subset \{1,\ldots,R\} \\ |J|=k}} \prod_{i \in J} \lambda_i = x^R + \sum_{1 \le k \le R} x^{R-k} \, (-1)^k \, \sigma_k
$$
while on the other hand,
$$
P(x) = \prod_{i=1}^R (x-\lambda_i) = \det(xI-B) = x^R + \sum_{k=1}^R x^{R-k} \, (-1)^k \sum_{\substack{J \subset \{1,\ldots,R\}\\ |J|=k}} \det(B_J)
$$
so identifying the coefficients, we obtain equality \eqref{eq:ev_det}. This implies that $\sigma_k$ is also a polynomial in $b_{l,m}^2$. Finally, for $s_l = \sum_{1 \le i \le R} \lambda_i^l$, we have the following recursion, also known as Newton-Girard's formula:
$$
s_l = \sum_{1 \le i \le l-1} (-1)^{i-1} \sigma_i \, s_{l-i} + (-1)^{l-1} \, l \, \sigma_l
$$
For example, $s_0=n$,$s_1=\sigma_1$, $s_2 = s_1 \, \sigma_1-2\sigma_2$, $s_3 = s_2 \, \sigma_1 - s_1 \sigma_2 + 3 \sigma_3$, etc. We therefore find by induction that for all $k$, $(-1)^k \, \mathrm{Tr}(B^{2k}) = (-1)^k \, s_{2k}$ is also a polynomial in $b_{l,m}^2$, but this alone does not guarantee the positivity of the coefficients, obtained in Lemma \ref{lem:3} above.
\end{rem}

%%%%%%%%%%%%%%%%%%%%%%%%%%%%%%%%%%%%%%%%%%%%%%%%%%%%%%%%%%%%%%%%

\subsection{Formulas regarding the eigenvalues and eigenvectors of $A(\mathbf{x})$ and $B(\mathbf{x},\mathbf{c})$}

We first state the following lemma, which has important consequences on the eigenvalues of the matrices $A(\mathbf{x})$ and $B(\mathbf{x},\mathbf{c})$, as highlighted hereafter. The approach taken below generalizes the method initiated by Montgomery and Vaughan in \cite{MV73}.

\begin{lem} \label{lem:5}
a) Let $\mathbf{u}=(u_1,\ldots,u_R)^T$ be an eigenvector of $A(\mathbf{x})$ for the eigenvalue $i\mu$.
Then for $1 \le n \le R$, we have
\begin{equation} \label{eq:mu_a}
\mu^2 \, \vert u_n\vert^2 = \sum_{1 \le m \le R} a_{m,n}^2 \, (\vert u_m\vert^2 + 2 \, \Re(u_n \, \overline{u_m}))
\end{equation}
b) Let $\mathbf{u}=(u_1,\ldots,u_R)^T$be an eigenvector of $B(\mathbf{x},\mathbf{c})$ for the eigenvalue $i\mu$.
Then for $1 \le n \le R$, we have
\begin{equation} \label{eq:mu_b}
\mu^2 \, \vert u_n \vert^2 = \sum_{1 \le m \le R} \, a_{m,n}^2 \, (c_n^2 \, c_m^2 \, \vert u_m \vert^2 + 2 \, c_n^3 \, c_m \, \Re(u_n \, \overline{u_m}))
\end{equation}
\end{lem}

\begin{proof}
Clearly, \eqref{eq:mu_a} is a particular case of \eqref{eq:mu_b} (with all $c_n=1$). In what follows, we prove \eqref{eq:mu_b} directly.

Our starting assumption is $B\mathbf{u}=i\mu\mathbf{u}$, i.e.~$\sum_{1 \le m \le R} b_{n,m} \, u_m = i\mu \, u_n$.
Taking the modulus square on both sides, we obtain
$$
\mu^2 \, \vert u_n \vert^2 =\sum_{\substack{1 \le l,m \le R\\ l \ne n, \, m \ne n}} b_{n,m} \, b_{n,l} \, u_m \, \overline{u}_l
$$
(Notice that the sum can be taken over $l \ne n, \, m \ne n$ as $b_{n,n}=0$). Therefore,
\begin{equation} \label{eq:s1s2}
\mu^2 \, \vert u_n \vert^2 = c_n^2 \sum_{\substack{1 \le l,m \le R \\ l \ne n, m \ne n}} c_l \, c_m \, a_{n,m} \, a_{n,l} \, u_m \, \overline{u}_l
= c_n^2 \, (S_1+S_2)
\end{equation}
where $S_1$ corresponds to the terms in the sum with $l=m$, i.e.
\begin{equation} \label{eq:s1}
S_1 =\sum_{\substack{1 \le m \le R \\ m \ne n}} c_m^2 \, a_{m,n}^2 \, \vert u_m \vert^2
\end{equation}
and
$$
S_2=\sum_{\substack{1 \le l,m \le R \\ l\ne m, \, l \ne n, \, m\ne n}} c_l \, c_m \, a_{n,m} \, a_{n,l} \, u_m \, \overline{u}_{l}
$$
As $l,m,n$ are all distinct in the above sum, we can use \eqref{eq:a} and the antisymmetry of $A$ gives
$$
a_{n,m} \, a_{n,l} = a_{l,m} \, a_{n,l} + a_{m,l} \, a_{n,m}
$$
so
\begin{equation} \label{eq:s3s4}
S_2 = \sum_{\substack{1 \le l,m \le R \\ l \ne m, \, l \ne n, \, m \ne n}} c_l \, c_m \, (a_{l,m} \, a_{n,l} + a_{m,l} \, a_{n,m}) \, u_m \, \overline{u}_l = S_3 + S_4 
\end{equation}
with
\begin{eqnarray*}
S_3 & = & \sum_{\substack{1 \le l,m \le R \\ l \ne m, \, l \ne n, \, m \ne n}} c_l \, c_m \, a_{l,m} \, a_{n,l} \, u_m \, \overline{u}_l
= \sum_{\substack{1 \le l,m \le R \\ l \ne m, \, l \ne n, \, m \ne n}} b_{l,m} \, a_{n,l} \, u_m \, \overline{u}_l
= \sum_{\substack{1 \le l \le R\\ l \ne n}} a_{n,l} \, \overline{u}_l \sum_{\substack{1 \le m \le R\\ m \ne l, \; m \ne n}} b_{l,m} \, u_m
\end{eqnarray*}
As $\mathbf{u}$ is an eigenvector of $B$, it follows that
$$
S_3 = \sum_{\substack{1 \le l \le R \\ l \ne n}} a_{n,l} \, \overline{u}_l \, (i\mu \, u_l - b_{l,n} u_n)
$$
Likewise, noticing that $\overline{\mathbf{u}}$ is also an eigenvector of $B$ (with corresponding eigenvalue $-i\mu$), we obtain
$$
S_4 = \sum_{\substack{1 \le m \le R \\ m \ne n}} a_{n,m} \, u_m \sum_{\substack{1 \le l \le R \\ l \ne n}} b_{m,l} \, \overline{u}_l
= \sum_{\substack{1 \le m \le R \\ m \ne n}} a_{n,m} \, u_m \, (-i\mu\overline{u}_m - b_{m,n} \, \overline{u}_n)
$$
From \eqref{eq:s3s4}, we deduce that
$$
S_2 = S_3 + S_4 = - \sum_{\substack{1 \le m \le R \\ m \ne n}} a_{n,m} \, b_{m,n} \, (\overline{u}_m \, u_n +  u_m \, \overline{u}_n)
= 2 \sum_{\substack{1 \le m \le R \\ m \ne n}} a_{m,n} \, b_{m,n} \, \Re(u_m \, \overline{u}_n)
$$
Now, using this together with \eqref{eq:s1s2} and \eqref{eq:s1}, we finally obtain
$$
\mu^2 \, \vert u_n \vert^2 = \sum_{\substack{1 \le m \le R \\ m \ne n}} c_n^2 \, \left( c_m^2 \, a_{m,n}^2 \, \vert u_m \vert^2
+ 2 \, c_m \, c_n \, a_{m,n}^2 \, \Re(u_m \, \overline{u}_n) \right)
$$
which completes the proof.
\end{proof}

One of the many consequences of Lemma \ref{lem:5} is the following.

\begin{cor} \label{cor:1}
If $c_1,\ldots,c_R$ are all non-zero, then the eigenvalues of $B(\mathbf{x},\mathbf{c})$ are all distinct.
\end{cor}

\begin{proof}
Indeed, if in the basis of eigenvectors of $B$, there were two eigenvectors corresponding to the same eigenvalue, then it would be possible to find a linear combination of them (which is also an eigenvector) such that one component (say $u_n$) would be equal to zero. Then by \eqref{eq:mu_b}, we would have
$$
\sum_{1 \le m \le R} a_{m,n}^2 \, c_n^2 \, c_m^2 \vert u_m \vert^2=0
$$
which is impossible, given the assumption made.
\end{proof}

A more precise version of Lemma \ref{lem:5}.b) reads as follows.

\begin{lem} \label{lem:6}
Let $\mathbf{u}=\mathbf{v}+i\mathbf{w}$ be an eigenvector of $B(\mathbf{x},\mathbf{c})$ $($with $\mathbf{v}=\Re(\mathbf{u})$, $\mathbf{w}=\Im(\mathbf{u})$) corresponding to the eigenvalue $i\mu$, then
\begin{equation} \label{eq:6temp}
\mu^2 \, v_n^2 = \sum_{1 \le m \le R} b_{n,m}^2 \, w_m^2 + 2 \, c_n^2 \sum_{\substack{1 \le m \le R \\ m \ne n}} a_{n,m} \, w_m \, (\mu \, v_m \, -b_{m,n} \, w_n)
\end{equation}
Moreover, if $\mu \ne 0$, then $\Vert \mathbf{v} \Vert = \Vert \mathbf{w} \Vert$, while if $\mu=0$, then $\det(B)=0$, so one of the eigenvectors corresponding to this eigenvalue is real.
\end{lem}

\begin{proof}
Applying the proof method of Lemma \ref{lem:5} gives
$$
\mu^2 \, v_n^2 = \Big(\sum_{1 \le m \le R} b_{n,m} \, w_{m} \Big)^2 = \sum_{1 \le m \le R} b_{n,m}^2 \, w_m^2 + \sum_{\substack{1 \le l,m \le R \\ l \ne m}} b_{n,m} \, b_{n,l} \, w_m \, w_l = S_1 + S_2
$$
with
$$
S_1 = \sum_{1 \le m \le R} b_{n,m}^2 \, w_m^2
$$
and
$$
S_2=\sum_{\substack{1 \le l,m \le R \\ l \ne m}} b_{n,m} \, b_{n,l} \, w_m \, w_l = c_n^2 \sum_{\substack{1 \le l,m \le R\\ l \ne m}} \, c_l \, c_m \, a_{n,m} \, a_{n,l} \, w_m \, w_l = c_n^2 \, (S_3+S_4)
$$
with again
\begin{eqnarray*}
S_3 & = & \sum_{\substack{1 \le l,m \le R \\ l \ne m, \, l \ne n, \, m \ne n}} \, c_l \, c_m \, a_{l,m} \, a_{n,l} \, w_m \, w_l =
\sum_{\substack{1 \le l \le R\\ l \ne n}} a_{n,l} w_l \, \sum_{\substack{1 \le m \le R \\ m \ne n, \, m \ne l}} b_{l,m} \, w_m\\
& = & \sum_{\substack{1 \le l \le R \\ l \ne n}} a_{n,l} \, w_l \, (\mu \, v_l - b_{l,n} \, w_n)
\end{eqnarray*}
and likewise,
$$
S_4 = \sum_{\substack{1 \le m \le R \\ m \ne n}} a_{n,m} \, w_m \, \sum_{\substack{1 \le l \le R \\ l \ne m, l \ne n}} b_{m,l} \, w_l
= \sum_{\substack{1 \le m \le R\\ m \ne n}} a_{n,m} \, w_m \, (\mu \, v_m - b_{m,n} w_n)
$$
Observing that $S_3=S_4$, we obtain the formula \eqref{eq:6temp}.

Finally, we have by assumption $B(\mathbf{v}+i\mathbf{w}) = i\mu \, (\mathbf{v}+i\mathbf{w})$, so
$$
B \, \mathbf{w} = \mu \, \mathbf{v} \quad \text{and} \quad B \, \mathbf{v} = -\mu \, \mathbf{w}
$$
Consequently,
$$
\mu \, \Vert \mathbf{w} \Vert^2 = \mu \, \mathbf{w}^T \, \mathbf{w}
= (-B \, \mathbf{v})^T \, \mathbf{w} =  (B^T \, \mathbf{v})^T \, \mathbf{w} =
\mathbf{v}^T \, B \, \mathbf{w} = \mu \, \Vert \mathbf{v} \Vert^2
$$
so for $\mu \ne 0$, we have $\Vert \mathbf{v} \Vert = \Vert \mathbf{w} \Vert$.
\end{proof}

Finally, let us mention the following nice formula.

\begin{lem} \label{lem:7}
Let $\mathbf{u}$ be an eigenvector of $B$ corresponding to the eigenvalue $\mu$,  Then
$$
\Big| \sum_{1 \le r \le R} c_r \, u_r \Big|^2 = \sum_{1 \le r \le R} \vert c_r \, u_r \vert^2
$$
\end{lem}

\begin{proof}
Let $C=\mathrm{diag}(c_1,\ldots,c_R)$ and $X=\mathrm{diag}(x_1,\ldots,x_R)$. Then
$$
\overline{\mathbf{u}}^T \, (XCAC-CACX) \, \mathbf{u}=\overline{\mathbf{u}}^T \, M \, \mathbf{u}
$$
where $m_{r,s} = c_r \, c_s$ for $r \ne s$ and $0$ otherwise. Therefore,
$$
\overline{\mathbf{u}}^T \, M \, \mathbf{u} = \Big| \sum_{1 \le r \le R} c_ r \, u_r \Big|^2 - \sum_{1 \le r \le R} | c_r \, u_r |^2
$$
On the other hand,
$$
\overline{\mathbf{u}}^ T \, (XCAC-CACX) \, \mathbf{u} = \overline{\mathbf{u}}^T \, (XB-BX) \, \mathbf{u} = \overline{\mathbf{u}}^T \, X \, i \mu \, \mathbf{u} - i \mu \, \overline{\mathbf{u}}^T \, X \, \mathbf{u} = 0
$$
as $\overline{\mathbf{u}}^T (-B) = \overline{\mathbf{u}}^T \, B^T = (B \, \overline{\mathbf{u}})^T = (-i\mu \overline{ \mathbf{u}})^T = -i\mu \, \overline{\mathbf{u}}^T$. The result follows.
\end{proof}
%CONSEQUENCE?

%%%%%%%%%%%%%%%%%%%%%%%%%%%%%%%%%%%%%%%%%%%%%%%%%%%%%%%%%%%%%%%%%

\subsection{Back to the spectral norm}
Lemma \ref{lem:5} also allow us to deduce the following bounds on the spectral norm of $A(\mathbf{x})$.

\begin{cor} \label{cor:2}
\begin{equation} \label{eq:7}
\max_{1 \le m \le R} \sum_{1\le n\le R} a_{m,n}^2 \le \Vert A(\mathbf{x}) \Vert^2 \le 3 \, \max_{1 \le m \le R} \sum_{1 \le n \le R}a_{m,n}^2 
\end{equation}
\end{cor}

\begin{proof}
The left-hand side inequality is clear, as the $m^{th}$ column of $A$ is the image by $A$ of the $m^{th}$ canonical vector. For the right-hand side inequality, we use \eqref{eq:mu_b}, choosing $n$ such that $\vert u_n \vert^2 \ge \vert u_m \vert^2$ for all $1 \le m \le R$, and $\mu=\Vert A \Vert$. We therefore obtain
$$
\Vert A \Vert^2 \, \vert u_n\vert^2 = \sum_{1 \le m \le R} a_{m,n}^2 (\vert u_m \vert^2 + 2\Re(u_n \, \overline{u_{m}})) \le \sum_{1 \le m \le R} a_{m,n}^2 (\vert u_m \vert^2+\vert u_m\vert^2+\vert u_n\vert^2)
$$
so
$$
\Vert A \Vert^2 \, \vert u_n \vert^2 \le 3 \sum_{1 \le m \le R} a_{m,n}^2 \, \vert u_n\vert^2
$$
\end{proof}

%%%%%%%%%%%%%%%%%%%%%%%%%%%%%%%%%%%%%%%%%%%%%%%%%%%%%%%%%%%%%%%%%%

\subsection{The classical Hilbert matrix $T_R$}
First of all, notice that the upper bound in equation \eqref{eq:7} allows to recover to the original upper bound on $\Vert T_R \Vert$, where $T_R$ is the Hilbert matrix defined in the introduction:
$$
\Vert T_R \Vert^2 \le \max_{1 \le m \le R} 3 \sum_{1 \le n \le R, \, n \ne m} \frac{1}{(m-n)^2} < 3 \cdot 2 \sum_{n \ge 1} \frac{1}{n^2} = \pi^2
$$
Let us now come back to the convergence speed of $\Vert T_R \Vert$ towards $\pi$, already mentioned in Section \ref{sec:survey}. We shall now prove inequality \eqref{eq:conv_speed}, namely the fact that there exist positive constants $a$ and $b$ such that
$$
\frac{a}{R} < \pi - \Vert T_R \Vert < \frac{b \, \log(R)}{R} \quad (R \ge 2)
$$
The lower bound can be deduced from Lemma \ref{lem:5}. From equation \eqref{eq:mu_b}, we indeed see that if $R=2S+1$, then
$$
\Vert T_R \Vert^2 < 6\sum_{k=1}^S \frac{1}{k^2} = \pi^2 - 6 \sum_{k>S} \frac{1}{k^2} < \pi^2-6\sum_{k>S} \frac{1}{k(k+1)} = \pi^2 -\frac{6}{S+1}
$$
so
$$
\pi - \Vert T_R \Vert > \frac{6}{(S+1) \, (\pi+\Vert T_R \Vert)} >\frac{3}{\pi \, (S+1)}
$$
which is indeed of the type $\frac{a}{R} < \pi - \Vert T_R \Vert$.

Another way to prove this lower bound is to follow the Grenander-Szeg\"o approach of Section \ref{sec:toeplitz}. Let us first recall equation \eqref{eq:fourier}:
$$
\mathbf{u}^* T_R \mathbf{u} = \int_0^{2\pi} f(x) \, |\phi(x)|^2 \, dx
$$
where $f(x) = i \, (x-\pi)$ for $x \in (0,2\pi)$ and $\phi(x)=\frac{1}{\sqrt{2\pi}} \sum_{1\le n \le R} u_n \, \exp(i(n-1)x)$, where both $\int_0^{2\pi} |\phi(x)|^2 \, dx = \Vert \mathbf{u} \Vert^2 =1$. Therefore,
\begin{equation} \label{eq:lowerbound}
\pi - \mathbf{u}^* iT_R \mathbf{u} = \int_0^{2\pi} x \, |\phi(x)|^2 \, dx \quad \text{so} \quad
\pi - \Vert T_R \Vert = \inf_{\phi \in E(R)} \int_0^{2\pi} x \, |\phi(x)|^2 \, dx
\end{equation}
where
$$
E(R) = \left\{\phi(x)= \frac{1}{\sqrt{2\pi}} \sum_{1 \le n \le R} u_n \, \exp(i (n-1) x) \, \bigg| \, {\mathbf u} \in \CC^R, \, \sum_{1 \le n \le R} |u_n|^2 = 1 \right\}
$$
It remains therefore to show that the term on the right-hand side of \eqref{eq:lowerbound} is bounded below by a term of order $1/R$. To this end, let us consider $\phi \in E(R)$
and $c>0$:
\begin{eqnarray*}
\int_0^c |\phi(x)|^2 \, dx & = & \frac{1}{2\pi} \sum_{1 \le m,n \le R} u_m \, \overline{u_n} \int_0^c \exp(i (m-n)x) \, dx\\
& \le & \frac{c}{2\pi} \sum_{1 \le m,n \le R} |u_m| \, |u_n| = \frac{c}{2\pi} \, \left( \sum_{1 \le n \le R} 1 \, |u_n| \right)^2 \le \frac{c R}{2\pi} \, \sum_{1 \le n \le R}|u_n|^2 = \frac{cR}{2\pi} 
\end{eqnarray*}
where we have used Cauchy-Schwarz' inequality. Setting therefore $c=\frac{\pi}{R}$, we obtain $\int_0^{\pi/R} |\phi(x)|^2 \, dx \le \frac{1}{2}$. This in turn implies that for all $\phi \in E(R)$,
$$
\int_0^{2\pi} x \, |\phi(x)|^2 \, dx \ge \int_{\pi/R}^{2\pi} x \, |\phi(x)|^2 \, dx \ge \frac{\pi}{R} \int_{\pi/R}^{2\pi}|\phi(x)|^2 \, dx \ge \frac{\pi}{2R}
$$
which settles the lower bound in equation \eqref{eq:conv_speed}.

In order to establish the upper bound, we need to find a function $\phi \in E(R)$ such that
\begin{equation} \label{eq:upperbound}
\int_0^{2\pi} x \, |\phi(x)|^2 \, dx \le \frac{b \log R}{R}
\end{equation}
for some constant $b>0$. This will indeed ensure the existence of a vector ${\mathbf u}$, the one associated to the function $\phi \in E(R)$, such that $|{\mathbf u}^* T_R {\mathbf u}| \ge \pi - \frac{b \log R}{R}$, implying the result.

In view of equation \eqref{eq:upperbound}, our goal in the following is to find $\phi \in E(R)$ such that for both $c$ and $\varepsilon$ small,
\begin{equation} \label{eq:phi_delta}
\int_c^{2\pi} |\phi(x)|^2 \, dx \le \varepsilon
\end{equation}
This would indeed imply that
\begin{equation} \label{eq:phi_delta2}
\int_0^{2\pi} x \, |\phi(x)|^2 \, dx \le c \int_0^c |\phi(x)|^2 \, dx + 2 \pi \int_c^{2\pi} |\phi(x)|^2 \, dx \le c + 2\pi \varepsilon
\end{equation}
Let $M$ and $N$ be positive integers such that $N(M-1)+1 \le R$ and let
$$
g(x) = \Big( \sum_{0 \le m \le M-1} \exp(i m x)  \Big)^N
$$
The function $\phi$ defined as $\phi(x)=g(x-c/2) \Big/ \sqrt{\int_0^{2\pi} |g(x)|^2 \, dx}$ belongs to $E(R)$. Our claim is that for an appropriate choice of $M$ and $N$, $\phi$ satisfies \eqref{eq:phi_delta} with both $c$ and $\varepsilon$ small.

We first show the following estimate on $\int_0^{2\pi} |g(x)|^2 \, dx$.

\begin{lem} \label{lem:8}
$$
\frac{M^{2N}}{N(M-1)+1} \le \frac{1}{2\pi} \int_0^{2\pi} |g(x)|^2 \, dx \le M^{2N-1}
$$
\end{lem}

\begin{proof}
Let $K$ be a positive integer and define the polynomial
$$
P_K(t) = \Big( \sum_{0 \le m \le M-1} t^m \Big)^K = \sum_{0 \le l \le K(M-1)} b_{l,K} \, t^l
$$
Notice that clearly, $b_{l,K} = b_{m,K}$ if $l+m=K(M-1)$. Moreover,
$$
|g(x)|^2 = \left|P_N(\exp(i x)) \right|^2 = \sum_{0 \le l,m \le N(M-1)} b_{l,N} \, b_{m,N} \, \exp(i (l-m) x)
$$
so
$$
\int_0^{2\pi} |g(x)|^2 \, dx = 2 \pi \sum_{0 \le l \le N(M-1)} b_{l,N}^2 =  2 \pi \sum_{0 \le l \le N(M-1)} b_{l,N} \, b_{N(M-1)-l,N} = 2 \pi \, b_{N(M-1),2N}
$$
What remains therefore to be proven is
$$
\frac{M^{2N}}{N(M-1)+1} \le b_{N(M-1),2N} \le M^{2N-1}
$$
Using Cauchy-Schwarz' inequality, we obtain
$$
b_{N(M-1),2N} = \sum_{0 \le l \le N(M-1)} b_{l,N}^2 \ge \frac{\left(\sum_{0 \le l \le N(M-1)} b_{l,N} \right)^2}{N(M-1)+1}=\frac{P_N(1)^2}{N(M-1)+1}=\frac{M^{2N}}{N(M-1)+1}
$$
On the other hand, $P_{2N}(t) = P_1(t) \, P_{2N-1}(t)$, so
$$
b_{N(M-1),2N} = \sum_{(N-1) (M-1) \le l \le N (M-1)} b_{l,2N-1} \le P_{2N-1}(1) \le M^{2N-1} 
$$
which completes the proof.
\end{proof}

We now set out to prove \eqref{eq:phi_delta}. Recall that $\phi(x)=g(x-c/2)\Big/\sqrt{\int_0^{2\pi} |g(x)|^2 \, dx}$. As a result of the previous lemma,
$$
\int_c^{2\pi} |\phi(x)|^2 \, dx \le \frac{N(M-1)+1}{M^{2N}} \, \frac{1}{2\pi} \int_c^{2\pi} |g(x-c/2)|^2 \, dx
= \frac{N(M-1)+1}{M^{2N}} \, \frac{1}{2\pi} \int_{c/2}^{2\pi-c/2} |g(x)|^2 \, dx 
$$
Notice that
$$
|g(x)|^2 = \Big| \sum_{0 \le m \le M-1} \exp(imx) \Big|^{2N} = \left( \frac{\sin(Mx/2)}{\sin(x/2)} \right)^{2N}
$$
so
$$
\int_{c/2}^{2\pi-c/2} |g(x)|^2 \, dx = 2 \int_{c/2}^\pi |g(x)|^2 \, dx \le 2 \int_{c/2}^\pi \left( \frac{\pi \, \sin(Mx/2)}{x} \right)^{2N} \, dx
$$
as for $0 \le x \le \pi$, $\sin(x/2) \ge x/\pi$. This implies
$$
\int_{c/2}^{2\pi-c/2} |g(x)|^2 \, dx \le 2 \int_{c/2}^\infty \left( \frac{\pi}{x} \right)^{2N} \, dx
= 2 \pi \int_{c/2\pi}^\infty \frac{1}{y^{2N}} \, dy = \frac{2 \pi}{2N-1} \, \left(\frac{2\pi}{c}\right)^{2N-1}
$$
and correspondingly
$$
\varepsilon = \int_c^{2\pi} |\phi(x)|^2 \, dx \le \frac{N(M-1)+1}{M^{2N}} \, \frac{1}{2N-1} \, \left(\frac{2\pi}{c}\right)^{2N-1}
$$
Assuming $R \ge 3$ and defining $M:=\left[\frac{2R}{\log R} \right]$, $N:=\left[\frac{\log R}{2} \right]$ and $c:=\frac{\pi e \log R}{R}$ (where $[x]$ denotes the integer part of $x$), we verify that $M(N-1)+1 \le R$ (so $\phi \in E(R)$) and prove below that \eqref{eq:phi_delta} is satisfied with $\varepsilon=O(1/R)$. Indeed, as $M \ge \frac{R}{\log R}$ and $N(M-1)+1\le M(2N-1)$, we obtain
$$
\frac{N(M-1)+1}{M^{2N} \, (2N-1) \, (c/2\pi)^{2N-1}} = (cM/2\pi)^{1-2N} \, \frac{1+N(M-1)}{M(2N-1)} \le (cM/2\pi)^{1-2N} \le \exp(1-2N) \le \frac{e^3}{R}
$$
as $1-2N < 3 -\log R$. According to \eqref{eq:phi_delta2}, this finally leads to
$$
\int_0^{2\pi} x | \phi(x)|^2 \, dx \le \frac{\pi e \log R}{R} + \frac{2 \pi e^3}{R}
$$
which completes the proof of the upper bound in \eqref{eq:conv_speed}. As already mentioned, it has been conjectured in \cite{P85} that of the two bounds in \eqref{eq:conv_speed}, the upper bound is tight. We provide below some numerical simulation that supports this fact; on Figure \ref{fig:1}, the expression
$$
f(R) := (\pi - \Vert T_R \Vert) \, \frac{R}{\log R}
$$
is represented as a function of $R$, for values of $R$ ranging from $1$ to $10'000$:

\begin{figure}
\begin{center}
\includegraphics[width=6cm]{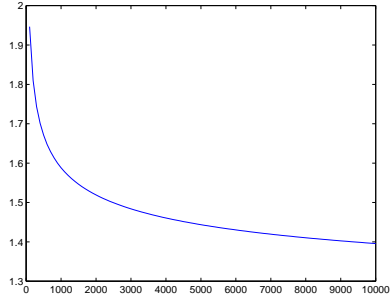}
\caption{Rescaled gap $f(R)$ between the spectral norm of the infinite-dimensional operator $T_\infty$ and that of the matrix $T_R$, as a function of $R \in \{1,\ldots,10'000\}$.}
\label{fig:1}
\end{center}
\end{figure}

Detailed facts can also be established about the eigenvectors of $T_R$. In order to ease the notation, suppose that $R=2S+1$ and that $T_R$ is indexed from $-S$ to $S$.

\begin{lem}
Let $\mathbf{u}$ be an eigenvector of $T_R$, corresponding to the eigenvalue $i\mu$, with $u_0=1$ (one can always multiply $\mathbf{u}$ by a constant in order to ensure that this is the case). Then for $0 \le n \le S$, we have 
$$
u_{-n}=-\overline{u_{n}}
$$
\end{lem}

\begin{proof}
Define $\mathbf{v}$ by $v_n=-\overline{u_{-n}}$. Then
$$
(T_R\mathbf{v})_{-m} = \sum_{-S \le n \le S} \frac{v_n}{-m-n} = \sum_{-S\le n\le S} \frac{v_{-n}}{-m+n}
= -\sum_{-S\le n\le S} \frac{v_{-n}}{m-n}
$$
So
$$
(T_{R}\mathbf{v})_{-m} = \sum_{-S\le n\le S}\frac{\overline{u_n}}{m-n} = (T_R\mathbf{\overline{u}})_m
= (-i\mu \,\mathbf{\overline{u}})_m = i\mu \, v_{-m}
$$
i.e.~$\mathbf{v}$ is an eigenvector corresponding to the eigenvalue $i\mu$, with $v_0=1$. Thus, $\mathbf{v}=\mathbf{u}$  (as  the eigenspace corresponding to $i\mu$ is of dimension $1$).
\end{proof}

We finally conjecture that the following statement holds. Let $\mathbf{u}$ be the eigenvector corresponding to the largest eigenvalue $\mu$ in absolute value. Then
$$
\vert u_m \vert < \vert u_n\vert \quad \forall 0 \le m<n\le S
$$
This conjecture is confirmed numerically; on Figure \ref{fig:2}, we represent $|u_n|$ as a function of $n \in \{-S,\ldots,S\}$, for $S=1'000$.

\begin{figure}
\begin{center}
\includegraphics[width=6cm]{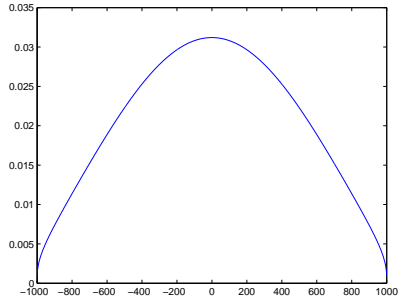}
\caption{Amplitude $\{|u_n|, \, -R \le n \le R\}$ of the eigenvector corresponding to the largest eigenvalue of $T_R$, with $R=1'000$.}
\label{fig:2}
\end{center}
\end{figure}

From the theoretical point of view, the above conjecture seems also reasonable, as $(-1)^k \, (T_R^{2k})_{n,n}$ (see Lemma \ref{lem:2}) should decrease as $n$ increases (in absolute value). If true, this fact would therefore hold in the limit $k \to \infty$, which would imply the conjecture on the eigenvector.

\end{document}